\documentclass{llncs}
\usepackage[utf8x]{inputenc}
\usepackage{amsmath}
\usepackage{enumerate}
\usepackage{algorithm}
\usepackage{algpseudocode}
\usepackage{ifpdf}

\ifpdf
\usepackage[pdftex,unicode]{hyperref}
\else
\usepackage[dvips,unicode]{hyperref}
\fi

\PrerenderUnicode{Õ√}

\title{Counting Square-Free Numbers}
\author{Jakub Pawlewicz}
\institute{Institute of Informatics\\ University of Warsaw\\ \email{pan@mimuw.edu.pl}}

\newcommand{\softO}{\tilde O}
\newcommand{\var}[1]{\textit{#1}}

\algnewcommand\algorithmicforeach{\textbf{for each}}
\algdef{S}[FOR]{ForEach}[1]{\algorithmicforeach\ #1\ \algorithmicdo}%

\begin{document}

\maketitle

\begin{abstract}
  The main topic of this contribution is the problem of counting square-free numbers not exceeding $n$.
  Before this work we were able to do it in time%
  \footnote{Comparing to the Big-O notation, Soft-O ($\softO$) ignores logarithmic factors.}
  $\softO(\sqrt{n})$.
  Here, the algorithm with time complexity $\softO(n^{2/5})$
  and with memory complexity $\softO(n^{1/5})$ is presented.
  Additionally, a parallel version is shown,
  which achieves full scalability.
  
  As of now the highest computed value was for $n=10^{17}$.
  Using our implementation we were able to calculate the value for $n=10^{36}$
  on a cluster.
\end{abstract}


\section{Introduction}

  A square-free number is an integer which is not divisible by a square of any integer greater than one.
  Let $S(n)$ denote the number of square-free positive integers less or equal to $n$.
  We can approximate the value of $S(n)$ using the asymptotic equation:
  \begin{equation*}
    S(n)=\frac{6}{\pi^2}n+O(\sqrt{n}).
  \end{equation*}
  Under the assumption of the Riemann hypothesis the error term can be further reduced \cite{jia}:
  \begin{equation*}
    S(n)=\frac{6}{\pi^2}n+O(n^{17/54+\varepsilon}).
  \end{equation*}
  Although these asymptotic equations allow us to compute $S(n)$ with high accuracy,
  they do not help to compute the exact value.
  
  The basic observation for efficient algorithms is the following formula.
  \begin{theorem}\label{thm:Sn}
    \begin{equation}\label{eq:Sn}
      S(n)=\sum_{d=1}^{\lfloor\sqrt{n}\rfloor}\mu(d)\Bigl\lfloor\frac{n}{d^2}\Bigr\rfloor\enspace,
    \end{equation}
    where $\mu(d)$ is the Möbius function.
  \end{theorem}
  The simple proof of this theorem using the inclusion-exclusion principle
  is presented in App. \ref{sec:proof-Sn}.
  The same proof can be found in \cite{michon}.
  It allows the author to develop an $\softO(\sqrt{n})$ algorithm and
  to compute $S(10^{17})$.
  In Sect. \ref{sec:basic} we show details of this algorithm
  together with the reduction of the memory complexity to $O(\sqrt[4]{n})$.
  
  To construct a faster algorithm we have to play with the summation \eqref{eq:Sn}.
  In Sect. \ref{sec:Sn_eff} the new formula is derived and stated in
  Theorem \ref{thm:eff}.
  Using this theorem we are able to construct the algorithm
  working in time $\softO(n^{2/5})$.
  It is described in the rest of Sect. \ref{sec:eff}.
  However, to achieve a memory efficient procedure more research is required.
  The memory reduction problem is discussed in Sect. \ref{sec:memory},
  where the modifications leading to the memory complexity $\softO(n^{1/5})$
  are presented.
  The result is put into Algorithm \ref{alg:Sn_eff}.
  
  Applying Algorithm \ref{alg:Sn_eff} for huge $n$ leads to computing time measured in years.
  Therefore, a practical algorithm should be distributed.
  Section \ref{sec:parallel} addresses the parallelization problem.
  At first sight it looks that Algorithm \ref{alg:Sn_eff} can be easily distributed,
  but a deeper analysis uncovers new problems.
  We present a solution for these problems, and get a fully scalable method.
  As a practical evidence, we computed $S(10^e)$ for all integers $e\le 36$,
  whereas before, the largest known value of $S(n)$ was for $n=10^{17}$ \cite{michon,oeis}.
  For instance, the value $S(10^{36})$ was computed in 88 hours using 256 processors.
  The detailed computation results are attached in Sect. \ref{sec:results}.

\section{\texorpdfstring{The $\softO(\sqrt n)$ algorithm}{The Õ(√n) algorithm}}
\label{sec:basic}

  We simply use Theorem \ref{thm:Sn} to compute $S(n)$.
  In order to compute summation \eqref{eq:Sn} we need to find the values of $\mu(d)$ for $d=1,\dots,K$, where $K=\lfloor\sqrt{n}\rfloor$.
  This can be done in time $O(K\log\log K)$ and in memory $O(\sqrt{K})$
  using a sieve similar to the sieve of Eratosthenes \cite{rivat}.
  See App. \ref{sec:calc-mu} for a detailed description.
  This sieving algorithm tabulates values in blocks of size $B=\lfloor\sqrt{K}\rfloor$.
  We assume we have the function $\textsc{TabulateMöbiusBlock}$ such that
  the call $\textsc{TabulateMöbiusBlock}(a, b)$ outputs the array $\var{mu}$ containing the values of the Möbius function:
  $\mu(k)=\var{mu}[k]$ for each $k\in(a,b\,]$.
  This function works in time $O(b\log\log b)$ and in memory $O(\max(\sqrt{b},b-a))$.
  Now, to calculate $S(n)$, we split the interval $[1,K]$ into $O(\sqrt{K})$ blocks of size $B$.
  It is presented in Algorithm~\ref{alg:Sn_basic}.
  
  \begin{algorithm}
    \caption{Calculating $S(n)$ in time $\softO(\sqrt{n})$ and in memory $O(\sqrt[4]{n})$}
    \label{alg:Sn_basic}
    \begin{algorithmic}[1]
      \State $s\gets 0, b\gets 0, K\gets\Theta(\lfloor\sqrt{n}\rfloor), B\gets \lfloor\sqrt{K}\rfloor$
      \Repeat
        \State $a\gets b, b\gets\min(b+B, K)$
        \State \Call{TabulateMöbiusBlock}{$a, b$}
        \For{$k=a+1,\dots,b$}
	  \State $\displaystyle s\gets s+\var{mu}[k]\cdot\Bigl\lfloor\frac{n}{k^2}\Bigr\rfloor$
        \EndFor
      \Until{$a\ge K$}
      \State \Return $s$
    \end{algorithmic}
  \end{algorithm}

  Summarizing, the basic algorithm has $\softO(\sqrt{n})$ time complexity
  and $O(\sqrt[4]{n})$ memory complexity.

\section{The New Algorithm}
\label{sec:eff}

  The key point of discovering a faster algorithm is a derivation of a new formula from \eqref{eq:Sn} in Sect. \ref{sec:Sn_eff}.
  The new formula depends on the Mertens function (the Möbius summation function).
  Section \ref{sec:Mertens} explains how one may compute the needed values.
  Section \ref{sec:eff-alg} states the algorithm.
  In Sect. \ref{sec:eff-complexity} the optimal values of the algorithm parameters are estimated, and the resulting time complexity of $\softO(n^{2/5})$ is derived.

\subsection{Establishing the New Formula}
\label{sec:Sn_eff}

  To alter \eqref{eq:Sn} we break the sum.
  We split the summation range $[1,\lfloor\sqrt n\rfloor]$ into two smaller intervals
  $[1,D]$ and $(D,\lfloor\sqrt n\rfloor]$:
  \begin{equation*}
    S(n)=S_1(n)+S_2(n)\enspace,
  \end{equation*}
  where
  \begin{align*}
    S_1(n)&=\sum_{1\le d\le D}\mu(d)\Bigl\lfloor\frac{n}{d^2}\Bigr\rfloor\enspace,\\
    S_2(n)&=\sum_{d>D}\mu(d)\Bigl\lfloor\frac{n}{d^2}\Bigr\rfloor\enspace.
  \end{align*}
  We introduced a new variable $D$.
  Optimal value of this variable will be determined later.
  Sum $S_2(n)$ can be rewritten using Iverson's convention\footnote{%
    $[P]=\begin{cases}1\quad&\text{if }P\text{ is true}\enspace,\\
    0&\text{otherwise}\enspace.\end{cases}$
  }:
  \begin{equation}\label{eq:S2n_derive1}
    S_2(n)
    =\sum_{d>D}\mu(d)\Bigl\lfloor\frac{n}{d^2}\Bigr\rfloor
    =\sum_{d>D}\sum_i\Bigl[i=\Bigl\lfloor\frac{n}{d^2}\Bigr\rfloor\Bigr]i\mu(d)\enspace.
  \end{equation}
  The predicate in brackets transforms as follows:
  \begin{equation*}
    i=\Bigl\lfloor\frac{n}{d^2}\Bigr\rfloor
    \iff
    i\leq\frac{n}{d^2}<i+1
    \iff
    \biggl\lfloor\sqrt{\frac{n}{i+1}}\biggl\rfloor
    <d\leq
    \biggl\lfloor\sqrt{\frac{n}{i}}\biggl\rfloor\enspace.
  \end{equation*}
  To shorten the notation we introduce a new variable $I$ and a new sequence $x_i$:
  \begin{equation}\label{eq:xi_def}
    x_i=\biggl\lfloor\sqrt{\frac{n}{i}}\biggl\rfloor\text{ for }i=1,\ldots,I\enspace.
  \end{equation}
  The sequence $x_i$ should be strictly decreasing.
  To ensure this, it is enough to set $I$ such that
  \begin{align}
    \sqrt{\frac{n}{I-1}}-\sqrt{\frac{n}{I}}&\ge1\notag\\
    \sqrt{n}&\ge\frac{\sqrt{I}\sqrt{I-1}}{\sqrt{I}-\sqrt{I-1}}\notag\\
    \sqrt{n}&\ge\sqrt{I}\sqrt{I-1}(\sqrt{I}+\sqrt{I-1})\notag\\
    n&\ge I(I-1)(\sqrt{I}+\sqrt{I-1})^2\enspace.\label{eq:Ibound1}
  \end{align}
  Because 
  \begin{equation*}
    I(I-1)(\sqrt{I}+\sqrt{I-1})^2<I\cdot I\cdot(2\sqrt{I})^2=4I^3\enspace,
  \end{equation*}
  to satisfy \eqref{eq:Ibound1}, it is enough to set
  \begin{equation}\label{eq:Ibound}
    I\leq\sqrt[3]{\frac{n}{4}}\enspace.
  \end{equation}
  
  Suppose we set $I$ satisfying $\eqref{eq:Ibound}$.
  Now, we take $D=x_I$ and we use $x_i$ notation \eqref{eq:xi_def} in \eqref{eq:S2n_derive1}:
  \begin{equation}\label{eq:S2n_derive2}
    S_2(n)
    =\sum_{d>x_I}\sum_i[x_{i+1}<d\leq x_i]i\mu(d)
    =\sum_{1\leq i<I}i\sum_{x_{i+1}<d\leq x_i}\mu(d)\enspace.
  \end{equation}
  Finally, it is convenient to use the Mertens function:
  \begin{equation}\label{eq:mertens}
    M(x)=\sum_{1\leq i\leq x}\mu(i)=\sum_{i=1}^{\lfloor x\rfloor}\mu(i)\enspace,
  \end{equation}
  \noindent thus we simplify \eqref{eq:S2n_derive2} to:
  \begin{equation}\label{eq:S2n}
    S_2(n)
    =\sum_{1\leq i<I}i\bigl(M(x_i)-M(x_{i+1})\bigr)
    =\biggl(\sum_{1\leq i<I}M(x_i)\biggr)-(I-1)M(x_I)\enspace.
  \end{equation}
  Theorem \ref{thm:eff} summarizes the above analysis.
  \begin{theorem}\label{thm:eff}
    Let $I$ be a positive integer satisfying
    $\displaystyle I\leq\sqrt[3]{\frac{n}{4}}$.
    Let
    $\displaystyle x_i=\biggl\lfloor\sqrt{\frac{n}{i}}\biggl\rfloor$ for $i=1,\ldots,I$
    and $D=x_I$.
    Then $S(n)=S_1(n)+S_2(n)$, where
    \begin{align*}
      S_1(n)&=\sum_{d=1}^D\mu(d)\Bigl\lfloor\frac{n}{d^2}\Bigr\rfloor\enspace,\\
      S_2(n)&=\Biggl(\sum_{i=1}^{I-1}M(x_i)\Biggr)-(I-1)M(x_I)\enspace.
    \end{align*}
  \end{theorem}

\subsection{Computing Values of the Mertens Function}
\label{sec:Mertens}

  By applying the Möbius inversion formula to \eqref{eq:mertens}
  we can get a nice recursion for the Mertens function:
  \begin{equation}\label{eq:mertens_rec}
    M(x)=1-\sum_{d\geq 2}M\Bigl(\frac{x}{d}\Bigr)\enspace.
  \end{equation}
  Here, an important observation is that having all values $M(x/d)$ for $d\ge2$,
  we are able to calculate $M(x)$ in time $O(\sqrt x)$.
  This is because there are at most $2\sqrt x$ different integers of the form $\lfloor x/d\rfloor$,
  since $x/d<\sqrt x$ for $d>\sqrt x$.
  
\subsection{The Algorithm}
\label{sec:eff-alg}

  The simple algorithm exploiting the above ideas is presented in Algorithm \ref{alg:sqrfree-fast}.
  \begin{algorithm}[htb]
    \caption{Efficient counting square-free numbers}
    \label{alg:sqrfree-fast}
    \begin{algorithmic}[1]
      \State compute $S_1(n)$ and $M(d)$ for $d=1,\dots,D$
      \For{$i=I-1,\dots,1$}\label{li:compute_mxi_for}
	\State compute $M(x_i)$ by \eqref{eq:mertens_rec}\label{li:compute_mxi}
      \EndFor\label{li:compute_mxi_endfor}
      \State compute $S_2(n)$ by \eqref{eq:S2n}
      \State\Return $S_1(n)+S_2(n)$
    \end{algorithmic}
  \end{algorithm}
  
  To compute $M(x_i)$ (line \ref{li:compute_mxi}) we need the values $M(x_i/d)$ for $d\ge2$.
  If $x_i/d\le D$ then $M(x_i/d)$ was determined during the computation of $S_1(n)$.
  If $x_i/d>D$ then see that
  \begin{equation}\label{eq:xi_by_d}
    \Bigl\lfloor\frac{x_i}{d}\Bigr\rfloor
    =\biggl\lfloor\frac{\bigl\lfloor\sqrt{\frac{n}{i}}\bigr\rfloor}{d}\biggr\rfloor
    =\biggl\lfloor\frac{\sqrt{\frac{n}{i}}}{d}\biggr\rfloor
    =\biggl\lfloor\sqrt{\frac{n}{d^2i}}\biggr\rfloor=x_{d^2i}\enspace,
  \end{equation}
  thus $M(x_i/d)=M(x_j)$ for $j=d^2i$.
  Of course $j<I$, because otherwise $\sqrt{n/j}\leq D$.
  Observe that it is important to compute $M(x_i)$ in a decreasing order (line~\ref{li:compute_mxi_for}).

\subsection{The Complexity}
\label{sec:eff-complexity}

  Let us estimate the time complexity of Algorithm \ref{alg:sqrfree-fast}.
  Computing $S_1(n)$ has complexity $O(D\log\log D)$.
  
  Computing $M(x_i)$ takes $O(\sqrt{x_i})$ time.
  The entire for loop (line \ref{li:compute_mxi_for}--\ref{li:compute_mxi_endfor})
  has the time complexity:
  \begin{equation}\label{eq:effalg_time1}
    \sum_{i=1}^{I}O(\sqrt{x_i})
    =\sum_{i=1}^{I}O\Biggl(\sqrt{\sqrt{\frac{n}{i}}}\,\Biggr)
    =O\biggl(\sqrt[4]{n}\sum_{i=1}^I\frac{1}{\sqrt[4]i}\biggr)\enspace.
  \end{equation}
  Using the asymptotic equality
  \begin{equation*}
    \sum_{i=1}^I\frac{1}{\sqrt[4]i}=\Theta(I^{3/4})\enspace,
  \end{equation*}
  \eqref{eq:effalg_time1} rewrites to:
  \begin{equation*}
    O\bigl(n^{1/4}I^{3/4}\bigr)\enspace.
  \end{equation*}
  
  The computation of $S_2(n)$ is dominated by the for loop.
  Summarizing the time complexity of Algorithm \ref{alg:sqrfree-fast} is
  \begin{equation}\label{eq:effalg_time_DI}
    O\bigl(D\log\log D+n^{1/4}I^{3/4}\bigr)\enspace.
  \end{equation}
  We have to tune the selection of $I$ and $D$ to minimize the expression \eqref{eq:effalg_time_DI}.
  The larger $I$ we take the smaller $D$ will be,
  thus the parameters $I$ and $D$ are optimal when
  \begin{equation*}
    O(D\log\log D)=O\bigl(n^{1/4}I^{3/4}\bigr)\enspace.
  \end{equation*}
  This takes place for
  \begin{equation*}
    I=n^{1/5}(\log\log n)^{4/5}\enspace,
  \end{equation*}
  and then
  \begin{equation*}
    O(D\log\log D)=O\bigl(n^{1/4}I^{3/4}\bigr)=O(n^{2/5}(\log\log n)^{3/5})\enspace.
  \end{equation*}
  \begin{theorem}\label{thm:effalg_time}
    The time complexity of Algorithm \ref{alg:sqrfree-fast} is
    $O(n^{2/5}(\log\log n)^{3/5})=\softO(n^{2/5})$
    for $I=n^{1/5}(\log\log n)^{4/5}=\softO(n^{1/5})$.
  \end{theorem}
  
  The bad news are the memory requirements.
  To compute $M(x_i)$ values we need to remember $M(d)$ for all $d=1,\dots,D$,
  thus we need $O(D)=\softO(n^{2/5})$ memory.
  This is even greater memory usage than in the basic algorithm.
  In the next section we show how to overcome this problem.

\section{Reducing Memory}
\label{sec:memory}

  To reduce memory we have to process values of the Möbius function in blocks.
  This affects the computation of needed Mertens function values which were previously computed by the recursion \eqref{eq:mertens_rec} as described in Sect. \ref{sec:mem-blocks}.
  These values have to be computed in a more organized manner.
  Section \ref{sec:mem-updates} provides necessary utilities for that.
  Moreover in Sect. \ref{sec:mem-structs} some data structures are introduced in order to achieve a satisfying time complexity.
  Finally, Sect. \ref{sec:mem-alg} states the algorithm together with a short complexity analysis.

\subsection{Splitting into Blocks}
\label{sec:mem-blocks}

  We again apply the idea of splitting computations into smaller blocks.
  To compute $S_1(n)$ we need to determine $\mu(d)$ and $M(d)$ for $d=1,\dots,D$.
  We do it in blocks of size $B=\Theta(\sqrt{D})$ by calling procedure \textsc{TabulateMöbiusBlock}.
  That way we are able to compute $S_1(n)$,
  but to compute $S_2(n)$ we face to the following problem.
  
  We need to compute $M(x_i)$ for integer $i\in[1,I)$.
  Previously, we memorized all needed $M(1),\dots,M(D)$ values and used recursion \eqref{eq:mertens_rec}.
  Now, we do not have unrestricted access to values of the Mertens function.
  After processing a block $(a,b\,]$ we have only access to values $M(k)$ for $k\in(a,b\,]$.
  We have to utilize these values before we switch to the next block.
  If a value $M(k)$ occurs on the right hand side of the recursion \eqref{eq:mertens_rec} for $x=x_i$ for some $i\in[1,I)$,
  then we should make an update.

  The algorithm should look as follows.
  We start the algorithm by creating an array $\var{Mx}$:
  \begin{equation*}
    \var{Mx}[i]\gets 1\quad\text{for }i=1,\dots,I-1\enspace.
  \end{equation*}
  During the computation of $S_1(n)$ we determine $M(k)$ for some $k$.
  Then, for every $i\in[1,I)$ such that $M(k)$ occurs in the sum
  \begin{equation}\label{eq:sumM}
    \sum_{d\ge 2}M\Bigl(\frac{x_i}{d}\Bigr)\enspace,
  \end{equation}
  i.e. for every $i\in[1,I)$ such that there exists an integer $d\ge 2$ such that 
  \begin{equation}\label{eq:ik_property}
    k=\Bigl\lfloor\frac{x_i}{d}\Bigr\rfloor\enspace,
  \end{equation}
  we estimate the number of occurrences $m$ of $M(k)$ in \eqref{eq:sumM} and update
  \begin{equation}\label{eq:Mx_update}
    \var{Mx}[i]\gets \var{Mx}[i]-m\cdot M(k)\enspace.
  \end{equation}
  After processing all $k=1,\dots,D$,
  there remains to update $\var{Mx}[i]$ by $M(x_i/d)$ for all $\lfloor x_i/d\rfloor >D$.
  With the help of equality \eqref{eq:xi_by_d} it is enough to update $\var{Mx}[i]$ by $M(x_{d^2i})$ for all $d^2i<I$.
  After these updates we will have $\var{Mx}[i]=M(x_i)$.

\subsection{Dealing with \texorpdfstring{$\var{Mx}$}{Mx} Array Updates}
\label{sec:mem-updates}

  The problem is how to, for given $k$, quickly find all possible values of $i$, that there exists an integer $d\ge 2$ fulfilling \eqref{eq:ik_property}.
  There is no simple way to do it in expected constant time.
  Instead, for given $i$ we can easily calculate successive~$k$.
  \begin{lemma}
    \label{lem:next_k}
    Suppose that for a given integer $i\in[1,I)$ and an integer $k$
    there exists an integer $d$ satisfying \eqref{eq:ik_property}.
    Let us denote
    \begin{align*}
      d_a&=\Bigl\lfloor\frac{x_i}{k}\Bigr\rfloor\enspace,\\
      d_b&=\Bigl\lfloor\frac{x_i}{k+1}\Bigr\rfloor\enspace,
    \end{align*}
    then
    \begin{enumerate}[(i)]
      \item
	the number of occurrences $m$, needed for update \eqref{eq:Mx_update}, equals $d_a-d_b$,
      \item
	the next integer $k$ satisfying \eqref{eq:ik_property} is for $d=d_b$,
	and it is equal to $\lfloor x_i/d_b\rfloor$.
    \end{enumerate}
  \end{lemma}
  \begin{proof}
    All possible integers $d$ satisfying \eqref{eq:ik_property} are:
    \begin{equation}
      k\le\frac{x_i}{d}<k+1
      \iff
      \frac{x_i}{k+1}<d\le\frac{x_i}{k}
      \iff
      \Bigl\lfloor\frac{x_i}{k+1}\Bigr\rfloor<d\le\Bigl\lfloor\frac{x_i}{k}\Bigr\rfloor\enspace,
    \end{equation}
    so \eqref{eq:ik_property} is satisfied for $d\in(d_b,d_a]$,
    and the next $k$ satisfying \eqref{eq:ik_property} is for $d=d_b$.
    \qed
  \end{proof}
  Lemma \ref{lem:next_k}, for every $i$,
  allows us to walk through successive values of $k$,
  for which we have to update $\var{Mx}[i]$.
  Since the target is to reduce the memory usage,
  we need to group all updates into blocks.
  Algorithm \ref{alg:Mx_update} shows how to utilize Lemma~\ref{lem:next_k} in order to update $\var{Mx}[i]$ for the entire block $(a,b\,]$.
  \begin{algorithm}[htb]
    \caption{Updating $\var{Mx}[i]$ for a block $(a,b\,]$}
    \label{alg:Mx_update}
    \begin{algorithmic}[1]
      \Require bounds $0\le a < b$, index $i\in[1,I)$, the smallest $k\in(a,b\,]$ that there exists $d$ satisfying \eqref{eq:ik_property}
      \Ensure $\var{Mx}[i]$ is updated by all $M(k)$ for $k\in(a,b\,]$,
	the smallest $k>b$ for the next update is returned
      \Function{MxBlockUpdate}{$a,b,i,k$}
	\State $\displaystyle d_a\gets\Bigl\lfloor\frac{x_i}{k}\Bigr\rfloor$
	\Repeat
	  \State $\displaystyle d_b\gets\Bigl\lfloor\frac{x_i}{k+1}\Bigr\rfloor$
	  \State $\var{Mx}[i]\gets \var{Mx}[i] - (d_a-d_b)\cdot M(k)$
	  \State $\displaystyle k\gets\Bigl\lfloor\frac{x_i}{d_b}\Bigr\rfloor$
	  \State $d_a\gets d_b$
	\Until{$k>b$}
	\State\Return $k$
      \EndFunction
    \end{algorithmic}
  \end{algorithm}

\subsection{Introducing Additional Structures}
\label{sec:mem-structs}

  Let $B=\lfloor\sqrt{D}\rfloor$ be the block size,
  and $L=\lceil D/B\rceil$ be the number of blocks.
  We process $k$ values in blocks $(a_0,a_1],(a_1,a_2],\dots,(a_{L-1},a_L]$,
  where $a_l=Bl$ for $0\le l<L$ and $a_L=D$.
  We need additional structures to keep track for every $i\in[1,I)$ where is the next update:
  \begin{itemize}
    \item $\var{mink}[i]$ stores the next smallest $k$ for which $\var{Mx}[i]$ has to be updated,
    \item $\var{ilist}[l]$ is a list of indexes $i$ for which the next update will be for $k$ belonging to the block $(a_l,a_{l+1}]$.
  \end{itemize}
  Using these structures we are able to perform every update in constant time.
  Once we update $\var{Mx}[i]$ for all necessary $k\in(a_l,a_{l+1}]$ by $\textsc{MxBlockUpdate}$,
  we get next $k>a_{l+1}$ for which the next update should be done.
  We can easily calculate the block index $l'$ for this $k$ and schedule it by putting $i$ into $\var{ilist}[l']$.

\subsection{The Algorithm}
\label{sec:mem-alg}

  The result of the entire above discussion is presented in Algorithm \ref{alg:Sn_eff}.
  We managed to preserve the number of operations,
  therefore the time complexity remained $\softO(n^{2/5})$.
  Each of the additional structures has $I=\softO(n^{1/5})$ or $L=O(\sqrt{D})=\softO(n^{1/5})$ elements.
  Blocks have size $O(B)=O(\sqrt{D})=\softO(n^{1/5})$,
  Therefore, the memory complexity of Algorithm \ref{alg:Sn_eff} is $\softO(n^{1/5})$.
  \begin{algorithm}[htbp]
    \caption{Calculating $S(n)$ in time $\softO(n^{2/5})$ and in memory $\softO(n^{1/5})$}
    \label{alg:Sn_eff}
    \begin{algorithmic}[1]
      \State $\displaystyle I\gets\Theta(n^{1/5}(\log\log n)^{4/5}),  D\gets\biggl\lfloor\sqrt{\frac{n}{I}}\biggr\rfloor,
      B\gets \lfloor\sqrt{D}\rfloor, L\gets\lceil D/B\rceil$
      \label{li:init-beg}
      \For{$l=0,\dots,L-1$}
        \State $\var{ilist}[l]\gets\emptyset$
      \EndFor
      \For{$i=0,\dots,I-1$}
	\State $\var{Mx}[i]\gets 1$
	\State $\var{mink}[i]\gets 1$
	\State $\var{ilist}[0]\gets \var{ilist}[0]\cup\{i\}$
      \EndFor
      \State $s_1\gets 0$
      \label{li:init-end}
      \For{$l=0,\dots,L-1$}\label{li:main_loop-beg}
      \Comment{blocks processing loop}
        \State \Call{TabulateMöbiusBlock}{$a_l, a_{l+1}$}
        \For{$k\in(a_l,\dots,a_{l+1}]$}
	  \State $\displaystyle s_1\gets s_1+\var{mu}[k]\cdot\Bigl\lfloor\frac{n}{k^2}\Bigr\rfloor$
        \EndFor
        \State compute $M(k)$ for $k\in(a_l,a_{l+1}]$ from values $\var{mu}[k]$ and $M(a_l)$
        \ForEach{$i\in \var{ilist}[l]$}
	  \State $\var{mink}[i]\gets$ \Call{MxBlockUpdate}{$a_l,a_{l+1},i,\var{mink}[i]$}
	  \State $\displaystyle l'\gets\Bigl\lfloor\frac{\var{mink}[i]}{B}\Bigl\rfloor$
	  \Comment{next block where $\var{Mx}[i]$ has to be updated}
	  \If{$l'\le L$ and $\var{mink}[i]<x_i$}
	    \State $\var{ilist}[l']\gets \var{ilist}[l']\cup\{i\}$
	  \EndIf
        \EndFor
        \State $\var{ilist}[l]\gets\emptyset$
      \EndFor\label{li:main_loop-end}
      \For{$i=I-1,\dots,1$}\label{li:mx_update-beg}
      \Comment{updating $\var{Mx}[i]$ by $M(k)$ for $k>D$}
        \ForAll{$d\ge 2$ such that $d^2i<I$}
          \State $\var{Mx}[i]\gets \var{Mx}[i] - \var{Mx}[d^2i]$
        \EndFor
      \EndFor\label{li:mx_update-end}
      \State compute $s_2=S_2(n)$ by \eqref{eq:S2n}\label{li:s2}
      \State \Return $s_1+s_2$
    \end{algorithmic}
  \end{algorithm}

  Observe that most work is done in the blocks processing loop (lines \ref{li:main_loop-beg}--\ref{li:main_loop-end}),
  because every other part of the algorithm takes at most $\softO(n^{1/5})$ operations.
  Initialization of structures (lines \ref{li:init-beg}--\ref{li:init-end}) is proportional to their size $\softO(n^{1/5})$.
  Computing $S_2(n)$ by \eqref{eq:S2n} (line \ref{li:s2}) takes $O(I)=\softO(n^{1/5})$.
  Only the time complexity of the part responsible for updating $\var{Mx}[i]$ by $M(k)$ for $k>D$ (lines \ref{li:mx_update-beg}--\ref{li:mx_update-end}) is unclear.
  The total number of updates in this part is:
  \begin{equation*}
    \sum_{i=1}^{I-1}\sum_{\substack{2\le d\\d^2i<I}}1
    \le\sum_{i=1}^I\sqrt{\frac{I}{i}}
    =\sqrt{I}\cdot\sum_{i=1}^I\frac{1}{\sqrt{i}}
    =\sqrt{I}\cdot O(\sqrt{I})
    =O(I)\enspace,
  \end{equation*}
  thus it is $O(I)=\softO(n^{1/5})$.

\section{Parallelization}
\label{sec:parallel}

  As noted in Sect. \ref{sec:mem-alg}, the most time consuming part of Algorithm \ref{alg:Sn_eff} is the blocks processing loop.
  The basic idea is to distribute calculations made by this loop between $P$ processors.
  We split the interval $[1,D]$ into a list of $P$ smaller intervals:
  $(a_0, a_1], (a_1, a_2], \dots, (a_{P-1}, a_P]$,
  where $0=a_0<a_1<\dots<a_P=D$.
  Processor number $p$, $0\le p< P$, focus only on the interval $(a_p,a_{p+1}]$,
  and it is responsible for
  \begin{enumerate}[(i)]
    \item\label{li:parallel_s1}
      calculating part of the sum $S_1(n)$
      \begin{equation}\label{eq:parallel_s1}
        \sum_{k\in(a_p,a_{p+1}]}\mu(k)\cdot\Bigl\lfloor\frac{n}{k^2}\Bigr\rfloor\enspace,
      \end{equation}
    \item\label{li:parallel_Mx}
      making updates of the array $\var{Mx}[1,\dots,I-1]$ for all $k\in(a_p,a_{p+1}]$.
  \end{enumerate}
  All processors share $s_1$ value and $\var{Mx}$ array.
  The only changes are additions of an integer,
  and it is required that these changes are atomic.
  Alternatively, a processor can collect all changes in its own memory,
  and, in the end, it only once change the value $s_1$ and each entry of $\var{Mx}$ array.

  Although the above approach is extremely simple, there are two drawbacks.
  First, for updates (\ref{li:parallel_Mx}),
  a processor needs to calculate successive values of the Mertens function:
  $M(a_p+1),\dots,M(a_{p+1})$.
  Computation of \eqref{eq:parallel_s1} produce successive values of the Möbius function starting from $\mu(a_p+1)$,
  therefore the Mertens function values can be also computed if only we knew the value of $M(a_p)$.
  Unfortunately, there is no other way than computing it from scratch.
  However, to compute $M(x)$ there is an algorithm working in time $\softO(x^{2/3})$ and memory $\softO(x^{1/3})$.
  See for instance \cite{rivat},
  or \cite{farey} for a simpler algorithm missing a memory reduction.

  In our application we have $x\le D=\softO(n^{2/5})$,
  therefore cumulative additional time we spend in computing Mertens function values from scratch is
  $\softO(PD^{2/3})=\softO(Pn^{4/15}).$
  We want this does not exceed the targeted time of $\softO(n^{2/5})$,
  therefore the number of processors is limited by:
  \begin{equation}
    P=\softO\biggl(\frac{n^{2/5}}{n^{4/15}}\biggr)=\softO(n^{2/15})\enspace.
  \end{equation}
  
  Second drawback comes from an observation that
  the number of updates of $\var{Mx}$ array is not uniformly distributed
  on $k\in[1,D]$.
  For example for $k\le\sqrt{D}=\softO(n^{1/5})$ for every $i\in[1,I)$ there always exists $d\ge2$ such that \eqref{eq:ik_property} is satisfied,
  therefore for every such $k$ there will be $I-1=\softO(n^{1/5})$ updates.
  It means that in a very small block $(1,\lfloor\sqrt{D}\rfloor]$ there will be $\softO(n^{2/5})$ updates, which is proportional to the total number of updates.
  We see that splitting into blocks is non-trivial and we need better tools for measuring work in the blocks processing loop.

  Let $t_s$ be the average time of computing a single summand of the sum $S_1(n)$, and let $t_u$ be the average time of a single update of $\var{Mx}$ array entry.
  Consider a block $(0,a]$.
  Denote as $U(a)$ the number of updates which must be done in this block.
  Then the expected time of processing this block is
  \begin{equation}\label{eq:work}
    T(a)=t_sa+t_uU(a)\enspace.
  \end{equation}
  It shows up that $U(a)$ can be very accurately approximated by a closed formula:
  \begin{equation}\label{eq:Ua}
    U(a)=
    \begin{cases}
      Ia\quad&\text{for }a\le\sqrt[4]{\frac{n}{I}}\enspace,\\
      \frac{1}{3}\frac{n}{a^3}-2\frac{n^{1/2}I^{1/2}}{a}+\frac{8}{3}n^{1/4}I^{3/4}
      &\text{for }\sqrt[4]{\frac{n}{I}}<a\le D=\sqrt{\frac{n}{I}}\enspace.
    \end{cases}
  \end{equation}
  See App. \ref{sec:Ua} for the estimation.
  
  The work measuring function \eqref{eq:work} says that
  the amount of work for the block $(a_p,a_{p+1}]$ is
  $T(a_{p+1})-T(a_p)$.
  Using this we are able to distribute blocks between processors in a such way,
  that the work is assigned evenly.
  
\section{Results}
\label{sec:results}

  We calculated $S(10^e)$ for all integer $0\le e\le 36$.
  In App. \ref{sec:Sn-power10} the computed values are listed.
  First, for $e\le 26$ we prepared the results using Algorithm \ref{alg:Sn_basic},
  the simpler and slower algorithm.
  Then we applied Algorithm \ref{alg:Sn_eff} on a single thread.
  Thus we verified its correctness for $e\le 26$ and we prepared further values for $e\le 31$.
  
  Finally, we used parallel implementation for $24\le e\le 36$.
  The computations were performed in ICM UW under grant G43-5 on the cluster \emph{Halo2}. See \cite{halo2} for a specification.
  The results for $e\le 31$ agreed with the previously prepared results.
  The timings of these computations are presented in Table \ref{tab:times}.
  \begin{table}[htbp]
  \centering
  \begin{center}
    \begin{tabular}{c@{\hspace{6pt}}r@{\hspace{6pt}}r@{\hspace{6pt}}r}
      & processors & computation & ideal \\
      $e$ & used & time & time \\
      \hline
      24 & 16 & 51 & 40 \\
      25 & 16 & 124 & 107 \\
      26 & 16 & 279 & 266 \\
      27 & 16 & 769 & 720 \\
      28 & 16 & 1928 & 1863 \\
      29 & 32 & 2594 & 2446 \\
      30 & 64 & 3439 & 3317 \\
      31 & 64 & 9157 & 8912 \\
      32 & 128 & 12138 & 11771 \\
      33 & 256 & 18112 & 16325 \\
      34 & 256 & 46540 & 43751 \\
      35 & 256 & 119749 & 115448 \\
      36 & 256 & 315313 & 303726
    \end{tabular}
  \end{center}
  \caption{Computation times in seconds of $S(10^e)$ for $24\le e\le 36$}
  \label{tab:times}
  \end{table}
  
  Computation time is calendar time in seconds of cluster occupation.
  Ideal time represents how long computations could take,
  if communication between processors was ignored
  and if the work was distributed equally.
  This was calculated by taking cumulative time of the actual work done for each processor and dividing by the number of processors.
  We see that ideal time is close to computation time showing an experimental evidence of scalability of the parallel algorithm.

\bibliographystyle{splncs03}
\bibliography{sqrfree}

\appendix

\section{Proof of Theorem \ref{thm:Sn}}
\label{sec:proof-Sn}

  Let our universe be all positive integers less or equal to $n$:
  \begin{equation}
    U=\{1,\ldots,n\}\enspace.
  \end{equation}
  For a prime integer $p$, let us define a set $A_p$:
  \begin{equation}
    A_p = \{ a\in U : p^2\text{ divides }a\}\enspace.
  \end{equation}
  Complement of set $A_p$ represents a set of all integers less or equal to $n$ not divisible by $p^2$.
  We want to count integers not divisible by any prime square,
  therefore the number we are searching for is the size of the set
  $\bigcap_{p\text{ prime}}\overline{A_p}$.
  By the inclusion-exclusion principle we have:
  \begin{align}
    \bigcap_{p\text{ prime}}\overline{A_p}\;=\;&
    |U|-\sum_{p\text{ prime}}|A_p|
    +\sum_{\substack{p<q\\p,q\text{ prime}}}|A_p\cap A_q|\notag\\
    & -\sum_{\substack{p<q<r\\p,q,r\text{ prime}}}|A_p\cap A_q\cap A_r|
    +\dots
    \notag\\
    =& \sum_{i=0}^\infty(-1)^i\sum_{\substack{p_1<\dots<p_i\\p_1,\dots,p_i\text{ prime}}}|A_{p_1}\cap\dots\cap A_{p_i}|\label{eq:inc-exc}
  \end{align}
  Now, observe that
  \begin{equation*}
    |A_{p_1}\cap\dots\cap A_{p_i}|=
    \biggl\lfloor\frac{n}{p_1^2\cdot\ldots\cdot p_i^2}\biggr\rfloor\enspace.
  \end{equation*}
  Using the Iverson bracket we can write \eqref{eq:inc-exc} as
  \begin{equation*}
    \eqref{eq:inc-exc}=
    \sum_d\sum_i[d=p_1\cdot\ldots\cdot p_i\land p_1<\dots<p_i\land p_1,\dots,p_i\text{ prime}](-1)^i\Bigl\lfloor\frac{n}{d^2}\Bigr\rfloor\enspace.
  \end{equation*}
  The expression $[d\text{ is a product of }i\text{ distinct primes}](-1)^i$ means the Möbius function $\mu(d)$ in other words,
  therefore we get the final formula of Theorem \ref{thm:Sn}.

\section{Computing the Möbius Function}
\label{sec:calc-mu}

  To compute values of the Möbius function we exploit the following property:
  \begin{equation}
    \mu(k)=
    \begin{cases}
      0\quad&\text{if $p^2$ divides k}\enspace,\\
      (-1)^e&\text{if }k=p_1\cdot\ldots\cdot p_e\enspace.
    \end{cases}
  \end{equation}
  Using a sieve we can find values of $\mu(k)$ for all $k=1,\ldots,K$ simultaneously,
  where $K=\sqrt{n}$, as presented in Algorithm \ref{alg:mobius_basic}.
  \begin{algorithm}[htb]
    \caption{Computing values of the Möbius function: the basic approach}
    \label{alg:mobius_basic}
    \begin{algorithmic}[1]
      \Require bound $1\le K$
      \Ensure $\mu(k)=\var{mu}[k]$ for $k=1,\dots,K$
      \Procedure{TabulateMöbius}{$K$}
	\For{$k=1,\dots,K$}
	  \State $\var{mu}[k]\gets 1$
	\EndFor
	\ForEach{prime $p\le K$}
	  \ForEach{$k\in[1,K]$ divisible by $p^2$}
	    \State $\var{mu}[k]\gets 0$
	  \EndFor
	  \ForEach{$k\in[1,K]$ divisible by $p$}
	    \State $\var{mu}[k]\gets -\var{mu}[k]$
	  \EndFor
	\EndFor
      \EndProcedure
    \end{algorithmic}
  \end{algorithm}
  To generate all primes less or equal to $K$ we can use the sieve of Eratosthenes.
  The memory complexity is $O(K)$ and the time complexity is $O(K\log\log K)$.

  The above method could be improved to fit $O(\sqrt{K})$ memory by tabulating in blocks.
  We split the array $\var{mu}$ to blocks of the size $B=\Theta(\sqrt{K})$,
  and for each block we tabulate $\mu(\cdot)$ separately using Algorithm \ref{alg:mobius_eff}.
  \begin{algorithm}[htb]
    \caption{Computing values of the Möbius function: memory efficient sieving in blocks}
    \label{alg:mobius_eff}
    \begin{algorithmic}[1]
      \Require bounds $0<a<b$
      \Ensure $\mu(k)=\var{mu}[k]$ for each $k\in(a,b\,]$
      \Procedure{TabulateMöbiusBlock}{$a,b$}
	\ForEach{$k\in(a,b\,]$}
	  \State $\var{mu}[k]\gets 1$
	  \State $m[k]\gets 1$
	  \Comment multiplicity of all found prime divisors of $k$
	\EndFor
	\ForEach{prime $p\le\sqrt{b}$}
	  \ForEach{$k\in(a,b\,]$ divisible by $p^2$}
	    \State $\var{mu}[k]\gets 0$
	  \EndFor
	  \ForEach{$k\in(a,b\,]$ divisible by $p$}
	    \State $\var{mu}[k]\gets -\var{mu}[k]$
	    \State $m[k]\gets m[k]\cdot p$
	  \EndFor
	\EndFor
	\ForEach{$k\in(a,b\,]$}
	  \If{$m[k]<k$}
	    \Comment $k=m[k]\cdot q$, where $q$ is prime and $q>\sqrt{b}$
	    \State $\var{mu}[k]\gets -\var{mu}[k]$
	  \EndIf
	\EndFor
      \EndProcedure
    \end{algorithmic}
  \end{algorithm}
  For each block we use only primes less or equal to $\sqrt{K}$,
  and we need only $O(\sqrt{K})$ memory.
  There is at most $K/B=O(\sqrt{K})$ blocks.
  Therefore, for each block the number of operations is
  \begin{equation}
    O(\sqrt{K})+\sum_{p\leq\sqrt{K}}\Bigl(1+\frac{B}{p^2}+\frac{B}{p}\Bigr)=
    O(\sqrt{K}+B\log\log K)=O(\sqrt{K}\log\log K)\enspace,
  \end{equation}
  which results in $O(K\log\log K)$ time complexity for the whole algorithm.

\section{Estimating \texorpdfstring{$U(a)$}{U(a)}}
\label{sec:Ua}

  Instead of computing the exact number of updates in a block $(0,a]$,
  we will compute an approximation of the expected number of updates as follows.
  Let us fix $k\in(0,a]$ and $x=x_i$ for some $i\in[1,I)$.
  The probability that there exists $d$, such that \eqref{eq:ik_property} is satisfied, equals
  \begin{equation*}
    P(k,x)=
    \begin{cases}
      1\quad&\text{for }k\le\sqrt{x}\enspace,\\
      \frac{x}{k}-\frac{x}{k+1}&\text{for }k>\sqrt{x}\enspace.
    \end{cases}
  \end{equation*}
  Let us define $U(a,x_i)$ as
  the expected number of updates of entry $\var{Mx}[i]$ for all $k\in(0,a]$.
  Let $x=x_i$.
  If $a\le\sqrt{x}$ then
  \begin{equation*}
    U(a,x)
    =\sum_{k=1}^aP(k,x)
    =\sum_{k=1}^a 1=a\enspace,
  \end{equation*}
  and if $a>\sqrt{x}$ then
  \begin{equation*}
    U(a,x)=
    \sum_{1\le k\le\sqrt{x}}1
    +\sum_{\sqrt{x}<k\le a}\Bigl(\frac{x}{k}-\frac{x}{k+1}\Bigr)
    =\lfloor\sqrt{x}\rfloor+\Bigl(\frac{x}{\lfloor\sqrt{x}\rfloor}-\frac{x}{a}\Bigr)
    \approx 2\sqrt{x}-\frac{x}{a}\enspace,
  \end{equation*}
  thus $U(a,x)$ can be presented as the formula:
  \begin{equation}\label{eq:Uax}
    U(a,x)=
    \begin{cases}
      a\quad&\text{for }a\le\sqrt{x}\enspace,\\
      2\sqrt{x}-\frac{x}{a}&\text{for }a>\sqrt{x}\enspace.
    \end{cases}
  \end{equation}
  Now we are ready to compute $U(a)$:
  \begin{equation}\label{eq:Ua_sum}
    U(a)=\sum_{1\le i<I}U(a,x_i)\enspace.
  \end{equation}
  Expanding a term $U(a,x_i)$ using \eqref{eq:Uax} depends on the inequality:
  \begin{equation*}
    a\le\sqrt{x_i}\iff
    a\le\sqrt{\sqrt{\frac{n}{i}}}\iff
    a\le\sqrt[4]{\frac{n}{i}}\iff
    i\le\frac{n}{a^4}\enspace.
  \end{equation*}
  Therefore $U(a,x_i)$ always expand to $a$ if $a\le\sqrt[4]{\frac{n}{I}}$,
  so then $U(a)=Ia$, and this is the first case of \eqref{eq:Ua}.
  Otherwise, if $a>\sqrt[4]{\frac{n}{I}}$, we split the summation \eqref{eq:Ua_sum}:
  \begin{align*}
    \sum_{1\le i<I}U(a,x_i)
    &=\sum_{1\le i\le\frac{n}{a^4}}U(a,x_i)
    +\sum_{\frac{n}{a^4}<i<I}U(a,x_i)\\
    &=\frac{n}{a^4}a+
    \sum_{\frac{n}{a^4}<i<I}\biggl(
    2\sqrt[4]{\frac{n}{i}}-\frac{1}{a}\sqrt{\frac{n}{i}}
    \,\biggr)\enspace.
  \end{align*}
  Now, we apply the following approximation formulas for sums:
  \begin{align*}
    \sum_{k=1}^x k^{-1/4}&\approx\frac{4}{3}x^{3/4}\enspace,\\
    \sum_{k=1}^x k^{-1/2}&\approx2x^{1/2}\enspace,
  \end{align*}
  and as a result we get the second case of \eqref{eq:Ua}:
  \begin{align*}
    U(a)&\approx
    \frac{n}{a^3}
    +2n^{1/4}\cdot\frac{4}{3}\biggl(I^{3/4}-\Bigl(\frac{n}{a^4}\Bigr)^{3/4}\biggr)
    -\frac{1}{a}n^{1/2}\cdot2\biggl(I^{1/2}-\Bigl(\frac{n}{a^4}\Bigr)^{1/2}\biggr)\\
    &=\frac{1}{3}\frac{n}{a^3}-2\frac{n^{1/2}I^{1/2}}{a}+\frac{8}{3}n^{1/4}I^{3/4}\enspace.
  \end{align*}

\section{\texorpdfstring{$S(n)$}{S(n)} values for powers of 10}
\label{sec:Sn-power10}

  \begin{table}[htbp]
  \centering
  \begin{center}
    \begin{tabular}{r|r@{}l}
      $e$ & & $S(10^e)$ \\
      \hline
      0 & & 1 \\
      1 & & 7 \\
      2 & & 61 \\
      3 & & 608 \\
      4 & & 6083 \\
      5 & & 60794 \\
      6 & & 607926 \\
      7 & & 6079291 \\
      8 & & 60792694 \\
      9 & & 607927124 \\
      10 & & 6079270942 \\
      11 & & 60792710280 \\
      12 & & 607927102274 \\
      13 & & 6079271018294 \\
      14 & & 60792710185947 \\
      15 & & 607927101854103 \\
      16 & & 6079271018540405 \\
      17 & & 60792710185403794 \\
      18 & & 607927101854022750 \\
      19 & & 6079271018540280875 \\
      20 & & 60792710185402613302 \\
      21 & & 607927101854026645617 \\
      22 & & 6079271018540266153468 \\
      23 & & 60792710185402662868753 \\
      24 & & 607927101854026628773299 \\
      25 & & 6079271018540266286424910 \\
      26 & & 60792710185402662866945299 \\
      27 & & 607927101854026628664226541 \\
      28 & & 6079271018540266286631251028 \\
      29 & & 60792710185402662866327383816 \\
      30 & & 607927101854026628663278087296 \\
      31 & & 6079271018540266286632795633943 \\
      32 & & 60792710185402662866327694188957 \\
      33 & & 607927101854026628663276901540346 \\
      34 & & 6079271018540266286632767883637220 \\
      35 & & 60792710185402662866327677953999263 \\
      36 & & 607927101854026628663276779463775476 \\
      \hline \\
      $\frac{6}{\pi^2}$ & 0. & $60792710185402662866327677925836583342615264\ldots$
    \end{tabular}
  \end{center}
  \caption{Values of $S(10^e)$ for $0\le e\le 36$}
  \label{tab:Sn-power10}
  \end{table}

\end{document}